\documentclass[12pt, reqno, a4paper]{amsart}
\usepackage[margin=1.2in]{geometry}
\numberwithin{equation}{section}
\usepackage{amssymb,amsfonts,amsthm}
\usepackage[utf8]{inputenc}
\usepackage{listings}
\usepackage{bm}
\usepackage[hyperpageref]{backref}
\usepackage{esint}
\usepackage{color}
\usepackage[bookmarks]{hyperref}
\usepackage{siunitx}
\usepackage{bigints}
\usepackage{hyperref}
\usepackage{mathtools}
\usepackage{mathrsfs}

\allowdisplaybreaks[4]

\newtheoremstyle{myremark}{10pt}{10pt}{}{}{\bfseries}{.}{.5em}{}

\newtheorem{theorem}{Theorem}[section]

\newtheorem{lemma}[theorem]{Lemma}

\theoremstyle{definition}

\newtheorem{remark}{Remark}

\usepackage{hyperref}

 \newcommand{\norm}[1]{\left\Vert#1\right\Vert}

\allowdisplaybreaks[4]

\begin{document}

\title{Fractional Sobolev logarithmic inequalities}

\author{Vivek Sahu}

\address{Department of Mathematics and Statistics, Indian Institute of Technology Kanpur, Kanpur - 208016, Uttar Pradesh, India}
\email{viveksahu20@iitk.ac.in, viiveksahu@gmail.com}



\date{}

\dedicatory{}

\begin{abstract}
We establish new  Euclidean Sobolev logarithmic inequalities in the framework of fractional Sobolev spaces and their weighted version. Our approach relies on a interpolation inequality, which can be viewed as a fractional Caffarelli–Kohn–Nirenberg type inequality. We further relate the optimal constant in this interpolation inequality to a corresponding variational problem. These results extend classical Sobolev logarithmic inequalities to the nonlocal Euclidean framework and provide new tools for analysis in fractional Sobolev spaces.
\end{abstract}

\maketitle


\section{Introduction}

This article focuses on proving a fractional version of the Euclidean $L^{p}$-Sobolev logarithmic inequality. In the local case, this inequality has been studied by many researchers. In particular, del Pino and Dolbeault \cite{Manuel2003} proved the Sobolev logarithmic inequality with the optimal constant.
They showed that, for $1 < p < d$, and for any $u \in W^{1,p}(\mathbb{R}^{d})$ with normalization $\int_{\mathbb{R}^{d}} |u(x)|^{p} \, dx = 1$, we have
\begin{equation}\label{log sob ineq in local}
    \int_{\mathbb{R}^{d}} |u(x)|^{p} \log |u(x)| \, dx \leq \frac{d}{p^{2}} \log \left[ C_{p} \int_{\mathbb{R}^{d}} |\nabla u(x)|^{p} \, dx \right],
\end{equation}
where 
\begin{equation}\label{best const C_p}
    C_{p} = \frac{p}{d} \left( \frac{p-1}{e}  \right)^{p-1} \pi^{- \frac{p}{2}} \left[ \frac{\Gamma \left(\frac{d}{2}+1 \right)}{\Gamma\left( d \frac{p-1}{p} +1 \right) } \right]^{\frac{p}{d}}
\end{equation}
is the best constant. Equality holds if and only if, for some $\sigma > 0$ and $\tilde{x} \in \mathbb{R}^{d}$,
\begin{equation*}
    u(x) = \pi^{\frac{d}{2}} \sigma^{-d \frac{(p-1)}{p}} \frac{\Gamma \left( \frac{d}{2}+1 \right)}{\Gamma \left( d \frac{p-1}{p} +1 \right)} e^{- \frac{1}{\sigma} |x - \tilde{x}|^{\frac{p}{p-1}}}, \quad  \forall \ x \in \mathbb{R}^{d}.
\end{equation*}
This Sobolev logarithmic inequality is very important in the theory of partial differential equations. Since the Poincaré inequality in Sobolev spaces does not hold in $\mathbb{R}^{d}$, the above result can be seen as a natural analogue of the Sobolev logarithmic inequality in this case. For more work related to logarithmic Sobolev inequalities, we refer to \cite{Adams1979, Beckner1998, Beckner1999, Carlen1991, Gross1975, Toscani1997, Weissler}.

\smallskip

Using the Sobolev inequality, one can also derive the Sobolev logarithmic inequality. However, in this case the constant obtained is not sharp. From the Sobolev inequality (see \cite{Beckner1999} and \cite[Page No. 153]{Manuel2003}), we get
    \begin{equation*}
    \int_{\mathbb{R}^{d}} |u(x)|^{p} \log |u(x)| \, dx \leq \frac{d}{p^{2}} \log \left[ D_{p} \int_{\mathbb{R}^{d}} |\nabla u(x)|^{p} \, dx \right], \quad \forall \ u \in W^{1,p}(\mathbb{R}^{d}),
\end{equation*}
where $D_{p}$ is Talenti’s constant \cite{Talenti1976}, given by
\begin{equation*}
    D_{p} = \frac{1}{d} \left( \frac{p-1}{d-p} \right)^{p-1} \pi^{- \frac{p}{2}} \left( \frac{\Gamma \left( d \right) \Gamma \left( \frac{d}{2} +1 \right)}{\Gamma \left( \frac{d}{p} \right) \Gamma \left( d \frac{p-1}{p} +1 \right) }  \right)^{\frac{p}{d}}.
\end{equation*}

\smallskip

In this article, we establish a fractional version of the Sobolev logarithmic inequality of the type \eqref{log sob ineq in local}. For $p \geq 1$ and $s \in (0,1)$, the fractional Sobolev space $W^{s,p}(\mathbb{R}^{d})$ is defined as the completion of compactly supported smooth functions under the norm $( \|\cdot\|^{p}_{L^{p}(\mathbb{R}^{d})} + [\cdot]^{p}_{W^{s,p}(\mathbb{R}^{d})})^{\frac{1}{p}}$, where 
\begin{equation*}
    [u]_{W^{s,p}(\mathbb{R}^{d})} := \left( \int_{\mathbb{R}^{d}} \int_{\mathbb{R}^{d}}  \frac{|u(x)-u(y)|^{p}}{|x-y|^{d+sp}} \, dx \, dy \right)^{\frac{1}{p}} 
\end{equation*}
is called the Gagliardo seminorm.  For more literature on fractional Sobolev spaces, we refer to the Hitchhiker’s guide \cite{Hitchiker2012}.

\smallskip

The following theorem presents the main result of this article. We establish a fractional version of the Sobolev logarithmic inequality in the case $sp < d$, with an explicit constant. Furthermore, by letting $s \to 1$, the result reduces to the classical Sobolev logarithmic inequality. The precise statement is given below:

\begin{theorem}\label{Theorem: Frac log Sob ineq}
    Let $p>1$ and $s \in (0,1)$ be such that $sp<d$. Then for every $u \in W^{s,p}(\mathbb{R}^{d})$ satisfying $\int_{\mathbb{R}^{d}} |u(x)|^{p} \, dx =1$, the following fractional Sobolev logarithmic inequality holds
    \begin{equation}\label{ineq in Theorem 1.1}
         \int_{\mathbb{R}^{d}} |u(x)|^{p} \log  |u(x)|   \, dx \leq \frac{d}{sp^{2}} \log \left( \mathcal{C} \int_{\mathbb{R}^{d}} \int_{\mathbb{R}^{d}}  \frac{|u(x)-u(y)|^{p}}{|x-y|^{d+sp}} \, dx \, dy \right),
    \end{equation}
    where
    \begin{equation}
        \mathcal{C}= C(d,p) \left( \frac{s(1-s)}{(d-sp)^{p-1}} \right),
    \end{equation}
   with $C(d,p) > 0$ depending only on $d$ and $p$.
\end{theorem}

Bourgain, Brezis, and Mironescu \cite{Brezis2001} established a fundamental link between fractional and classical Sobolev norms (see also \cite[Theorem~3.19 and Corollary~3.20]{Edmunds2023}). Specifically, for $p>1$ and $u \in L^{p}(\mathbb{R}^{d})$, there exists a constant $K=K(d,p)>0$ such that
\begin{equation}\label{eqn s to 1}
    \lim_{s \to 1} (1-s) \int_{\mathbb{R}^{d}} \int_{\mathbb{R}^{d}} \frac{|u(x)-u(y)|^{p}}{|x-y|^{d+sp}} \, dx \, dy = K \int_{\mathbb{R}^{d}} |\nabla u(x)|^{p} \, dx.
\end{equation}
Using this relation, and letting $s \to 1$ in inequality \eqref{ineq in Theorem 1.1} of Theorem~\ref{Theorem: Frac log Sob ineq}, we recover the classical Sobolev logarithmic inequality. In particular, we obtain
\begin{equation*}
    \int_{\mathbb{R}^{d}} |u(x)|^{p} \log  |u(x)|   \, dx \leq \frac{d}{p^{2}} \log \left( C(d,p) \frac{K}{(d-p)^{p-1}} \int_{\mathbb{R}^{d}} |\nabla u(x)|^{p} \, dx \right).
\end{equation*}
del Pino and Dolbeault \cite{Manuel2003} derived the same inequality with the sharp constant $C_{p}$ given in \eqref{best const C_p}. Consequently, we deduce the bound
\begin{equation*}
    C_{p} \leq C(d,p) \frac{K}{(d-p)^{p-1}}.
\end{equation*}

A version of the above theorem on stratified Lie groups was proved by Ghosh, Kumar, and Ruzhansky in \cite[Section $7$]{ghosh2023}. They established subelliptic fractional logarithmic Sobolev inequalities on stratified Lie groups with explicit constants depending on an additional parameter lying between $p$ and the fractional critical Sobolev exponent, under the condition $sp<d$. In contrast, Theorem \ref{Theorem: Frac log Sob ineq} provides a constant $\mathcal{C}$ in Euclidean space that is independent of any extra parameter and also recovers the classical Sobolev logarithmic inequality.

\smallskip

Lieb and Loss, in \cite{Lieb2001}, proved another version of the Sobolev logarithmic inequality. It says that for any $c>0$, we have
\begin{equation}
    \int_{\mathbb{R}^{d}} |u(x)|^{2} \log \left( \frac{|u(x)|^{2}}{\| u\|^{2}_{L^{2}(\mathbb{R}^{d})}} \right) \, dx + d(1+ \log c) \| u \|^{2}_{L^{2}(\mathbb{R}^{d})} \leq \frac{c^{2}}{\pi} \int_{\mathbb{R}^{d}} | \nabla u(x)|^{2} \, dx. 
\end{equation}
This inequality becomes an equality only when $u$ is a constant multiple of the Gaussian function $\exp \left(- \frac{\pi |x|^{2}}{2c^{2}} \right)$. Later, Chatzakou and Ruzhansky, in \cite{Ruzhansky2024}, proved a fractional version of this inequality. They showed that if $u \in W^{s,2}(\mathbb{R}^{d})$ with $0<s< \frac{n}{2}$ and $c>0$, then
\begin{align*}
    \int_{\mathbb{R}^{d}} |u(x)|^{2} \log \left( \frac{|u(x)|^{2}}{\| u\|^{2}_{L^{2}(\mathbb{R}^{d})}} \right) \, dx + \frac{d}{s}(1+ \log c)  \| u \|^{2}_{L^{2}(\mathbb{R}^{d})} \leq \frac{d e c^{2}}{2s} C_{d,s}  \| (- \Delta u)^{\frac{s}{2}} u \|^{2}_{L^{2}(\mathbb{R}^{d})},
\end{align*}
where 
\begin{equation*}
    C_{d,s} = \frac{\Gamma \left( \frac{d-2s}{2} \right)}{2^{2s} \pi^{s}\Gamma \left( \frac{d+2s}{2} \right) } \left( \frac{\Gamma \left( d \right)}{\Gamma \left( \frac{d}{2} \right)}  \right)^{\frac{2s}{d}}
\end{equation*}
and 
\begin{equation*}
    \| (- \Delta u)^{\frac{s}{2}} u \|^{2}_{L^{2}(\mathbb{R}^{d})} = \int_{\mathbb{R}^{d}} |\widehat{u}(\xi)|^{2} (4 \pi^{2} |\xi|^{2})^{s} d \xi.
\end{equation*}
Here, $\widehat{u}$ denotes the Fourier transform of $u$. 

\smallskip

In the next theorem, we extend the fractional Sobolev logarithmic inequality established in \cite{Ruzhansky2024} to the case $sp<d$, with an explicit constant. Moreover, by passing to the limit $s \to 1$, we recover the another form of classical Sobolev logarithmic inequality.

\begin{theorem}\label{Theorem:2 Frac log Sob ineq}
    Let $p>1$ and $s \in (0,1)$ be such that $sp<d$. Then for any $u \in W^{s,p}(\mathbb{R}^{d})$ and any real number $c>0$, the following inequality holds
    \begin{align}
     \int_{\mathbb{R}^{d}} |u(x)|^{p} \log \left( \frac{|u(x)|^{p}}{\|u\|^{p}_{L^{p}(\mathbb{R}^{d})}} \right)  \, dx \, + \, & \frac{d}{s} (1+ \log c) \|u\|^{p}_{L^{p}(\mathbb{R}^{d})} \nonumber \\ & \leq \frac{ d e^{p-1} c^{p} \, \mathcal{C}}{sp}  \int_{\mathbb{R}^{d}} \int_{\mathbb{R}^{d}}  \frac{|u(x)-u(y)|^{p}}{|x-y|^{d+sp}} \, dx \, dy  ,
\end{align}
where $\mathcal{C}$ is as defined in \eqref{Defn mathcal C}.
\end{theorem}

By applying the relation \eqref{eqn s to 1} in the above theorem, we recover the local case. In particular, taking the limit $s \to 1$ in the above theorem, we obtain the following another form of the Sobolev logarithmic inequality in the local setting:
\begin{align}
     \int_{\mathbb{R}^{d}} |u(x)|^{p} \log \left( \frac{|u(x)|^{p}}{\|u\|^{p}_{L^{p}(\mathbb{R}^{d})}} \right)  \, dx \, + \, & d (1+ \log c) \|u\|^{p}_{L^{p}(\mathbb{R}^{d})} \nonumber \\ & \leq \frac{ d e^{p-1} c^{p} \, K}{p} \left( \frac{C(d,p)}{(d-p)^{p-1}} \right)  \int_{\mathbb{R}^{d}} |\nabla u(x)|^{p} \, dx.
\end{align}

\smallskip

The next theorem establishes a fractional Sobolev logarithmic inequality in the weighted fractional Sobolev space. 
We consider the space $W^{s,p,\alpha}(\mathbb{R}^{d})$ as the closure of $C^{1}_{c}(\mathbb{R}^{d})$ with respect to the weighted norm 
$\|\cdot\|_{W^{s,p,\alpha}(\mathbb{R}^{d})}$ (see \eqref{Weighted norm definition} for the precise definition). 
Here $\alpha, \alpha_{1}, \alpha_{2} \in \mathbb{R}$ are parameters with $\alpha = \alpha_{1}+\alpha_{2}$ together with certain admissible bounds. 
The main result is stated as follows:

\begin{theorem}\label{Theorem: Weighted Frac log inequality 1}
    Let $p>1$, $s \in (0,1)$, and $\alpha_{1}, \alpha_{2} \in \mathbb{R}$ be such that $\alpha= \alpha_{1}+ \alpha_{2}$, $\alpha_{1}p, \, \alpha_{2}p \in (-d, sp)$, $ 
      0 \leq \alpha p < sp$, and $
       sp-\alpha p < d$. Then for every $u \in W^{s,p, \alpha}(\mathbb{R}^{d})$ satisfying $\int_{\mathbb{R}^{d}} |u(x)|^{p} \, dx =1$, the following weighted fractional Sobolev logarithmic inequality holds
    \begin{equation}
         \int_{\mathbb{R}^{d}} |u(x)|^{p} \log  |u(x)|   \, dx \leq \frac{d}{(s-\alpha)p^{2}} \log \left( \mathcal{C}_{\alpha} \int_{\mathbb{R}^{d}} \int_{\mathbb{R}^{d}}  \frac{|u(x)-u(y)|^{p}}{|x-y|^{d+sp}} |x|^{\alpha_{1} p} |y|^{\alpha_{2} p} \, dx \, dy \right),
    \end{equation}
    where $\mathcal{C}_{\alpha} = \mathcal{C}_{\alpha}(d,p,s,\alpha_{1}, \alpha_{2})>0$.
\end{theorem}

\smallskip



\smallskip

The key step in proving the above three theorems is to establish an interpolation inequality, or in other words, a fractional Caffarelli–Kohn–Nirenberg type inequality. In Section \ref{Proof of Frac log Sob ineq} and Section \ref{Proof of Weighted Frac log Sob ineq}, we will derive such an interpolation inequality with a constant. In the next theorem, we state an interpolation inequality in the fractional Sobolev space with an optimal constant $\mathscr{L}$. This optimal constant $\mathscr{L}$ is connected to the infimum of a variational problem, which we will discuss in Section \ref{Proof of CKN with opt}. For the literature on Caffarelli–Kohn–Nirenberg type inequalities, we refer to \cite{Weiwei2022, Caffarelli1984, Lu2021, Squassina2018}.
For $s \in (0,1)$, $p > 1$, and $q > p$, we define the space $\mathcal{D}^{p,q}_{s}$ as
\begin{equation}
   \mathcal{D}^{p,q}_{s} = (W^{s,p}(\mathbb{R}^{d}) \cap L^{q}(\mathbb{R}^{d})) \setminus \{ 0\}.
\end{equation}
The next interpolation theorem is stated as follows.

\begin{theorem}\label{Theorem: CKN with opt const}
Let $p>1$ and $s \in (0,1)$ be such that $sp<d$, $p<q<\frac{p(d-s)}{d-sp}$, and define $r = p \, \frac{q-1}{p-1}$. Then for every $u \in \mathcal{D}^{p,q}_{s}$, the following interpolation inequality holds with an optimal constant $\mathscr{L}$,
\begin{equation}
    \|u\|_{L^{r}(\mathbb{R}^{d})} \leq \mathscr{L} \, [u]^{a}_{W^{s,p}(\mathbb{R}^{d})} 
    \, \|u\|^{1-a}_{L^{q}(\mathbb{R}^{d})},
\end{equation}
where
\begin{equation}
    a = \frac{d(q-p)}{(q-1)\bigl(dp - (d-sp)q \bigr)},
\end{equation}
and $\delta := dp - (d-sp)q > 0$.
\end{theorem}

\smallskip

This article is organized as follows. In Section \ref{Proof of Frac log Sob ineq}, we prove Theorems \ref{Theorem: Frac log Sob ineq} and \ref{Theorem:2 Frac log Sob ineq}, and establish an interpolation inequality that plays a central role in these proofs. Section \ref{Proof of Weighted Frac log Sob ineq} is devoted to extending logarithmic Sobolev inequalities to weighted fractional Sobolev spaces, where we prove Theorem \ref{Theorem: Weighted Frac log inequality 1}. Finally, in Section \ref{Proof of CKN with opt}, we relate the optimal constant in the interpolation inequality to the infimum of a corresponding variational problem.

\section{Proof of fractional Sobolev logarithmic inequalities}\label{Proof of Frac log Sob ineq} 

In this section, we prove Theorem \ref{Theorem: Frac log Sob ineq} and Theorem \ref{Theorem:2 Frac log Sob ineq} by establishing the fractional Sobolev logarithmic inequalities for the case $sp<d$. The main tool is the fractional Sobolev inequality with an explicit constant, given by Maz’ya and Shaposhnikova in \cite[Theorem 1]{Mazya2002}. The constant $\mathcal{C}$ in Theorem \ref{Theorem: Frac log Sob ineq} is the same as the one in \cite[Theorem 1]{Mazya2002}.
Using this fractional Sobolev inequality, we derive an interpolation inequality with the constant $\mathcal{C}^{\frac{a}{p}}$. However, this constant $\mathcal{C}^{\frac{a}{p}}$ is not sharp. In Section \ref{Proof of CKN with opt}, we will show the relation between the best constant in the interpolation inequality and a corresponding variational problem.

\smallskip

Next, we recall the fractional Sobolev inequality in the case $sp < d$, where $p^{*}_{s} := \frac{dp}{d-sp}$. This result is also due to Maz’ya and Shaposhnikova \cite[Theorem 1]{Mazya2002}. For the case $p=2$, the optimal constant in fractional Sobolev inequality was established by Cotsiolis and Tavoularis in \cite{Cotsiolis20024}.

\begin{lemma}\label{Lemma: Frac Sob ineq}
    Let $d \geq 1$, $p \geq 1$ and $s \in (0,1)$ such that $sp<d$. Then for any $u \in W^{s,p}(\mathbb{R}^{d})$, there holds
    \begin{equation}
       \left( \int_{\mathbb{R}^{d}} |u(x)|^{p^{*}_{s}} \, dx \right)^{\frac{1}{p^{*}_{s}}} \leq \mathcal{C}^{\frac{1}{p}} \left( \int_{\mathbb{R}^{d}} \int_{\mathbb{R}^{d}} \frac{|u(x)-u(y)|^{p}}{|x-y|^{d+sp}} \, dx \, dy \right)^{\frac{1}{p}},
    \end{equation}
    where $p^{*}_{s} = \frac{dp}{d-sp}$, and
        \begin{equation}\label{Defn mathcal C}
        \mathcal{C}= C(d,p) \left( \frac{s(1-s)}{(d-sp)^{p-1}} \right),
    \end{equation}
     with $C(d,p)>0$ depending only on $d$ and $p$.
\end{lemma}

\smallskip

Now we derive an interpolation inequality in the fractional Sobolev space. For this, we make use of the above lemma. This result is an important step towards proving Theorem \ref{Theorem: Frac log Sob ineq} and Theorem \ref{Theorem:2 Frac log Sob ineq}.

\begin{theorem}\label{Theorem: CKN ineq without opt}
    Let $p>1$ and $s \in (0,1)$ be such that $sp<d$, $p<q < \tfrac{p(d-s)}{d-sp}$, and $r=p \tfrac{q-1}{p-1}$. Then for all $u \in \mathcal{D}^{p,q}_{s}$,
    \begin{equation}
        \|u\|_{L^{r}(\mathbb{R}^{d})} \leq \mathcal{C}^{\frac{a}{p}} [u]^{a}_{W^{s,p}(\mathbb{R}^{d})} \|u\|_{L^{q}(\mathbb{R}^{d})}^{1-a},
    \end{equation}
    where $\mathcal{C}$ is as defined in \eqref{Defn mathcal C}, and
    \begin{equation}
        a= \frac{d(q-p)}{(q-1)(dp-(d-sp)q)},
    \end{equation}
   with $\delta = dp - (d-sp)q > 0$.
\end{theorem}

\begin{proof}
It is easy to see that $q<r<p^{*}_{s}$. Next, we determine the interpolation parameter $a \in (0,1)$ by requiring
\begin{equation*}
\frac{1}{r} = \frac{a}{p^{*}_{s}} + \frac{1-a}{q}.
\end{equation*}
Solving for $a$, we obtain
\begin{equation*}
a = \frac{\tfrac{1}{r}-\tfrac{1}{q}}{\tfrac{1}{p^{*}_{s}}-\tfrac{1}{q}}
= \frac{\dfrac{q-r}{rq}}{\dfrac{q-p^{*}_{s}}{p^{*}_{s}q}}
= \frac{(q-r)p^{*}_{s}}{(q-p^{*}_{s})r}.
\end{equation*}
Since
\begin{equation*}
q-r = q-\frac{p(q-1)}{p-1}
= \frac{q(p-1)-p(q-1)}{p-1}
= \frac{p-q}{p-1},
\end{equation*}
and
\begin{equation*}
q-p^{*}_{s} = q-\frac{dp}{d-sp}
= \frac{q(d-sp)-dp}{d-sp}
= -\frac{dp-(d-sp)q}{d-sp},
\end{equation*}
we deduce
\begin{equation*}
a = \frac{p^{*}_{s}  \tfrac{p-q}{p-1}}{r  \big(-\tfrac{dp-(d-sp)q}{d-sp}\big)}
= \frac{p^{*}_{s}(q-p)(d-sp)}{(p-1)r \,(dp-(d-sp)q))}.
\end{equation*}
Since $p^{*}_{s}(d-sp)=dp$ and $(p-1)r = p(q-1)$, this reduces to
\begin{equation*}
a = \frac{d(q-p)}{(q-1)\,(dp-(d-sp)q)}.
\end{equation*}
Now we estimate $\|u\|_{L^{r}(\mathbb{R}^{d})}$. We write
\begin{equation*}
\|u\|^{r}_{L^{r}(\mathbb{R}^{d})}=\int_{\mathbb{R}^d}|u(x)|^r \, dx
=\int_{\mathbb{R}^d}|u(x)|^{ar}\,|u(x)|^{(1-a)r} \, dx.
\end{equation*}
Applying Hölder’s inequality with exponents
\begin{equation*}
 \nu=\frac{p^{*}_{s}}{ar}, \qquad \xi=\frac{q}{(1-a)r},   
\end{equation*}
we obtain
\begin{equation*}
\int_{\mathbb{R}^d} |u(x)|^{ar}|u(x)|^{(1-a)r} \, dx
\leq
\|u\|_{L^{p^{*}_{s}}}^{ar}\,\|u\|_{L^q}^{(1-a)r}.
\end{equation*}
Thus,
\begin{equation*}
\|u\|_{L^{r}(\mathbb{R}^{d})}\le \|u\|_{L^{p^{*}_{s}}(\mathbb{R}^{d})}^a\,\|u\|_{L^{q}(\mathbb{R}^{d})}^{1-a}.
\end{equation*}
Finally, using the fractional Sobolev inequality (see Lemma \ref{Lemma: Frac Sob ineq}), we conclude
\begin{equation*}
   \|u\|_{L^{r}(\mathbb{R}^{d})} \leq  \mathcal{C}^{\frac{a}{p}} [u]^{a}_{W^{s,p}(\mathbb{R}^{d})} \|u\|_{L^{q}(\mathbb{R}^{d})}^{1-a}. 
\end{equation*}
This proves the theorem.
\end{proof}

\smallskip

The following lemma was used by del Pino and Dolbeault in \cite{Manuel2003} to prove the Sobolev logarithmic inequality in the local case. They combined an interpolation inequality with this lemma to obtain their result. We also use the same lemma to establish the logarithmic Sobolev inequality in the non-local case. For completeness, we give the proof of the following lemma.

\begin{lemma}\label{Lemma on log ineq}
    Let $1\leq q<r< \infty$. Then for any $u \in L^{q}(\mathbb{R}^{d}) \cap L^{r}(\mathbb{R}^{d})$, we have
    \begin{equation}
        \int_{\mathbb{R}^{d}} |u(x)|^{q} \log \left( \frac{|u(x)|}{\| u \|_{L^{q}(\mathbb{R}^{d})}} \right) \, dx \leq \frac{r}{r-q} \| u\|^{q}_{L^{q}(\mathbb{R}^{d})} \log \left( \frac{\|u \|_{L^{r}(\mathbb{R}^{d})}}{\| u \|_{L^{q}(\mathbb{R}^{d})}} \right).
    \end{equation}
\end{lemma}
\begin{proof}
    By applying Hölder’s inequality with $\alpha(s) = \frac{q}{s} \left( \frac{r-s}{r-q} \right)$, $q \leq s \leq r$, we get
\begin{equation*}
    \| u \|_{L^{s}(\mathbb{R}^{d})} \leq \| u \|^{\alpha(s)}_{L^{q} (\mathbb{R}^{d})} \| u \|^{1-\alpha(s)}_{L^{r}(\mathbb{R}^{d})}.
\end{equation*}
Taking logarithm on both sides,
\begin{equation}\label{eqn 01}
    \log  \| u \|_{L^{s}(\mathbb{R}^{d})} \leq \alpha(s) \log  \| u \|_{L^{q}  (\mathbb{R}^{d})} + (1- \alpha(s)) \log \| u \|_{L^{r}(\mathbb{R}^{d})}.
\end{equation}
Now, define $\Phi(s) := \log  \| u \|_{L^{s}(\mathbb{R}^{d})} = \frac{1}{s} \log \left( \int_{\mathbb{R}^{d}} |u(x)|^{s} \, dx \right) $. Then, we have
\begin{equation}\label{Phi prime s}
    s \Phi'(s) = \frac{1}{\| u \|^{s}_{L^{s}(\mathbb{R}^{d})}} \int_{\mathbb{R}^{d}} |u(x)|^{s} \log \left( \frac{|u(x)|}{\|u\|_{L^{s}(\mathbb{R}^{d})}} \right) \, dx .
\end{equation}
Also, set $\theta (s) = \alpha(s) \log  \| u \|_{L^{q}  (\mathbb{R}^{d})} + (1- \alpha(s)) \log \| u \|_{L^{r}(\mathbb{R}^{d})}$. Then, 
\begin{align}\label{theta prime s}
    \theta'(s) & = \alpha'(s) \left( \log  \| u \|_{L^{q}  (\mathbb{R}^{d})} - \log \| u \|_{L^{r}(\mathbb{R}^{d})} \right) \nonumber \\ &  = \left( \frac{qr}{s^{2}(r-q)} \right) \log \left( \frac{\| u \|_{L^{r}(\mathbb{R}^{d})}}{\| u \|_{L^{q}  (\mathbb{R}^{d})}} \right) .
\end{align}
From inequality \eqref{eqn 01}, we have $\Phi(s) \leq \theta(s)$ for $q \leq s \leq r$, and equality when $s=q$. So, $\Phi'(q) \leq \theta'(q)$. Combining \eqref{Phi prime s} and \eqref{theta prime s} at $s=q$, we obtain
\begin{equation*}
     \int_{\mathbb{R}^{d}} |u(x)|^{q} \log \left( \frac{|u(x)|}{\| u \|_{L^{q}(\mathbb{R}^{d})}} \right) \, dx \leq \frac{r}{r-q} \| u\|^{q}_{L^{q}(\mathbb{R}^{d})} \log \left( \frac{\|u \|_{L^{r}(\mathbb{R}^{d})}}{\| u \|_{L^{q}(\mathbb{R}^{d})}} \right).
\end{equation*}
Hence, the lemma is proved.
\end{proof}

\bigskip

Now we prove Theorem \ref{Theorem: Frac log Sob ineq} by combining Theorem \ref{Theorem: CKN ineq without opt} with Lemma \ref{Lemma on log ineq}. Note that the constant $\mathcal{C}$ in our fractional Sobolev logarithmic inequality is not optimal.

\begin{proof}[\textbf{Proof of Theorem \ref{Theorem: Frac log Sob ineq}}]Let us assume that $sp<d, ~ p<q< \frac{p(d-s)}{d-sp}$, and $r=p \frac{q-1}{p-1}$.  From Theorem \ref{Theorem: CKN ineq without opt}, we have
\begin{equation*}
      \|u\|_{L^{r}(\mathbb{R}^{d})} \leq \mathcal{C}^{\frac{a}{p}} [u]^{a}_{W^{s,p}(\mathbb{R}^{d})} \|u\|_{L^{q}(\mathbb{R}^{d})}^{1-a}, \quad \forall \ u \in \mathcal{D}^{p,q}_{s}, 
    \end{equation*}
    where 
    \begin{equation*}
         a= \frac{d(q-p)}{(q-1)(dp-(d-sp)q)}.
    \end{equation*}
The above inequality can be equivalently written as
\begin{equation*}
   \frac{\|u\|_{L^{r}(\mathbb{R}^{d})}}{\|u\|_{L^{q}(\mathbb{R}^{d})}}  \leq \mathcal{C}^{\frac{a}{p}} \left( \frac{[u]_{W^{s,p}(\mathbb{R}^{d})}}{\|u\|_{L^{q}(\mathbb{R}^{d})}}  \right)^{a}.
\end{equation*}
Taking logarithms on both sides yields
\begin{equation}
  \frac{1}{a}  \log \left( \frac{\|u\|_{L^{r}(\mathbb{R}^{d})}}{\|u\|_{L^{q}(\mathbb{R}^{d})}} \right) - \log  \left( \mathcal{C}^{\frac{1}{p}}  \right) \leq \log \left(  \frac{[u]_{W^{s,p}(\mathbb{R}^{d})}}{\|u\|_{L^{q}(\mathbb{R}^{d})}}  \right).
\end{equation}
From Lemma \ref{Lemma on log ineq}, we further have
\begin{equation*}
\left( \frac{r-q}{r} \right) \frac{1}{a} \int_{\mathbb{R}^{d}} \frac{|u(x)|^{q}}{\|u\|^{q}_{L^{q}(\mathbb{R}^{d})}}  \log \left( \frac{|u(x)|}{\| u \|_{L^{q}(\mathbb{R}^{d})}} \right) \, dx  \leq   \frac{1}{a}  \log \left( \frac{\|u\|_{L^{r}(\mathbb{R}^{d})}}{\|u\|_{L^{q}(\mathbb{R}^{d})}} \right).
\end{equation*}
Combining the above two inequalities and letting $q \to p$, with the values of $r$ and $a$ given in Theorem \ref{Theorem: CKN ineq without opt}, we obtain
\begin{equation}\label{ineq taking log}
    \frac{sp}{d} \int_{\mathbb{R}^{d}} \frac{|u(x)|^{p}}{\|u\|^{p}_{L^{p}(\mathbb{R}^{d})}} \log \left( \frac{|u(x)|}{\|u\|_{L^{p}(\mathbb{R}^{d})}} \right)  \, dx -  \log  \left( \mathcal{C}^{\frac{1}{p}}  \right) \leq \log \left(  \frac{[u]_{W^{s,p}(\mathbb{R}^{d})}}{\|u\|_{L^{p}(\mathbb{R}^{d})}}  \right).
\end{equation}
Finally, imposing the normalization $\int_{\mathbb{R}^{d}} |u(x)|^{p} \, dx = 1$, we deduce
\begin{equation*}
     \int_{\mathbb{R}^{d}} |u(x)|^{p} \log |u(x)|   \, dx \leq \frac{d}{sp^{2}} \log \left( \mathcal{C} [u]^{p}_{W^{s,p}(\mathbb{R}^{d})} \right).
\end{equation*}
This completes the proof.
\end{proof}

\begin{proof}[\textbf{Proof of Theorem \ref{Theorem:2 Frac log Sob ineq}}] We begin with the assumption of Theorem \ref{Theorem:2 Frac log Sob ineq}. From inequality \eqref{ineq taking log} in the proof of Theorem \ref{Theorem: Frac log Sob ineq}, we know that
\begin{equation*}
    \frac{sp}{d} \int_{\mathbb{R}^{d}} \frac{|u(x)|^{p}}{\|u\|^{p}_{L^{p}(\mathbb{R}^{d})}} \log \left( \frac{|u(x)|}{\|u\|_{L^{p}(\mathbb{R}^{d})}} \right)  \, dx -  \log  \left( \mathcal{C}^{\frac{1}{p}}  \right) \leq \log \left(  \frac{[u]_{W^{s,p}(\mathbb{R}^{d})}}{\|u\|_{L^{p}(\mathbb{R}^{d})}}  \right).
\end{equation*}
This inequality can be rewritten as
\begin{equation*}
    \frac{sp}{d} \int_{\mathbb{R}^{d}} \frac{|u(x)|^{p}}{\|u\|^{p}_{L^{p}(\mathbb{R}^{d})}} \log \left( \frac{|u(x)|^{p}}{\|u\|^{p}_{L^{p}(\mathbb{R}^{d})}} \right)  \, dx  \leq \log  \left( \mathcal{C} \left(  \frac{[u]^{p}_{W^{s,p}(\mathbb{R}^{d})}}{\|u\|^{p}_{L^{p}(\mathbb{R}^{d})}}  \right) \right).
\end{equation*}
It is easy to check that for any $b>0$ and $x>0$, the following inequality holds true
\begin{equation*}
    \log x \leq bx - \log b -1.
\end{equation*}
Applying this inequality to the above inequality with $x= \mathcal{C} \left(  \frac{[u]^{p}_{W^{s,p}(\mathbb{R}^{d})}}{\|u\|^{p}_{L^{p}(\mathbb{R}^{d})}}  \right)$ and $b= e^{p-1} c^{p}>0$, we obtain
\begin{align*}
    \frac{sp}{d} \int_{\mathbb{R}^{d}} |u(x)|^{p} \log \left( \frac{|u(x)|^{p}}{\|u\|^{p}_{L^{p}(\mathbb{R}^{d})}} \right)  \, dx  & \leq \|u\|^{p}_{L^{p}(\mathbb{R}^{d})} \Big( e^{p-1} c^{p} \, \mathcal{C} \left(  \frac{[u]^{p}_{W^{s,p}(\mathbb{R}^{d})}}{\|u\|^{p}_{L^{p}(\mathbb{R}^{d})}}  \right) \\ & \hspace{6mm} -(\log (e^{p-1}c^{p}) +1) \Big) \\ & =  e^{p-1} c^{p} \, \mathcal{C}  [u]^{p}_{W^{s,p}(\mathbb{R}^{d})} -  p(\log c +1) \|u\|^{p}_{L^{p}(\mathbb{R}^{d})} .
\end{align*}
Therefore, we arrive at
\begin{equation*}
  \int_{\mathbb{R}^{d}} |u(x)|^{p} \log \left( \frac{|u(x)|^{p}}{\|u\|^{p}_{L^{p}(\mathbb{R}^{d})}} \right)  \, dx + \frac{d}{s} (1+ \log c) \|u\|^{p}_{L^{p}(\mathbb{R}^{d})}  \leq \frac{ d e^{p-1} c^{p} \, \mathcal{C}}{sp}   [u]^{p}_{W^{s,p}(\mathbb{R}^{d})}  .
\end{equation*}
This finishes the proof of the theorem.
\end{proof}

\section{Proof of Weighted fractional Sobolev logarithmic inequality}\label{Proof of Weighted Frac log Sob ineq} 

In this section, we extend the fractional Sobolev logarithmic inequalities to the weighted fractional Sobolev space and establish Theorem \ref{Theorem: Weighted Frac log inequality 1}. We begin by defining the weighted fractional Sobolev space. For $p>1$, $s \in (0,1)$, and $\alpha, \alpha_{1}, \alpha_{2} \in \mathbb{R}$ with $\alpha = \alpha_{1}+ \alpha_{2}$, the weighted Gagliardo seminorm is given by  
\begin{equation*}
    [u]_{W^{s,p, \alpha}(\mathbb{R}^{d})} := \left( \int_{\mathbb{R}^{d}} \int_{\mathbb{R}^{d}} \frac{ |u(x)-u(y)|^{p}}{|x-y|^{d+sp}} |x|^{\alpha_{1} p} |y|^{\alpha_{2} p}  \, dx \, dy \right)^{\frac{1}{p}} \leq \infty, \quad  \forall \ u \in L^{p}(\mathbb{R}^{d}).
\end{equation*}
The norm of this space is defined as  
\begin{equation}\label{Weighted norm definition}
    \|u\|_{W^{s,p, \alpha}(\mathbb{R}^{d})} = \left( \|u\|^{p}_{L^{p}(\mathbb{R}^{d})} + [u]^{p}_{W^{s,p, \alpha}(\mathbb{R}^{d})} \right)^{\frac{1}{p}}.
\end{equation}
Dipierro and Valdinoci \cite{Valdinoci2015} proved that if $\alpha_{1}p, \, \alpha_{2}p \in (-d, sp)$ and $\alpha_{1}p+ \alpha_{2}p > -d$, then for every $u \in C^{1}_{c}(\mathbb{R}^{d})$, the seminorm $[u]_{W^{s,p, \alpha}(\mathbb{R}^{d})} < \infty$. We then define the space $ W^{s,p, \alpha}(\mathbb{R}^{d})$ as the closure of $C^{1}_{c}(\mathbb{R}^{d})$ with respect to the norm $\| \cdot \|_{W^{s,p, \alpha}(\mathbb{R}^{d})}$.

\smallskip

 Nguyen and Squassina in \cite{Squassina2018} studied Caffarelli–Kohn–Nirenberg type inequalities in fractional Sobolev spaces and proved the following theorem. Let $d \geq 1$, $p>1$, $q \geq 1$, $\tau >0$, $a \in (0,1]$, and $\alpha, \beta, \gamma \in \mathbb{R}$ satisfy
\begin{equation*}
    \frac{1}{\tau}+ \frac{\gamma}{d} = a \left( \frac{1}{p} + \frac{\alpha-s}{d} \right) + (1-a) \left( \frac{1}{q} + \frac{\beta}{d} \right)  \ .
\end{equation*}
If $a>0$, assume also that, with $\gamma = a \sigma + (1-a) \beta$,
\begin{equation*}
    0 \leq \alpha - \sigma
\end{equation*}
and 
\begin{equation*}
    \alpha - \sigma \leq s \hspace{.5cm} \text{if} \hspace{.5cm} \frac{1}{\tau}+ \frac{\gamma}{d} = \frac{1}{p}+ \frac{\alpha-s}{d}  .
\end{equation*}
If $\frac{1}{\tau} + \frac{\gamma}{d} >0$, then we have
\begin{equation}\label{fractional CKN ineq}
          \| |x|^{\gamma}u \|_{L^{\tau}(\mathbb{R}^d)} \leq C [u]^{a}_{W^{s,p, \alpha}(\mathbb{R}^d)} \||x|^{\beta} u \|^{(1-a)}_{L^{q} (\mathbb{R}^d)},  \hspace{.3cm}  \forall \ u  \in C^{1}_{c}(\mathbb{R}^d)  .
      \end{equation}
\smallskip

Now, if we take $\gamma=0$ and $a=1$, the fractional Caffarelli--Kohn--Nirenberg inequality \eqref{fractional CKN ineq} reduces to  
\begin{equation}  \label{frac-sobolev}
\left( \int_{\mathbb{R}^{d}} |u(x)|^{\tau} \, dx \right)^{\!\frac{1}{\tau}}
 \leq C \left( \int_{\mathbb{R}^{d}} \int_{\mathbb{R}^{d}} 
 \frac{ |u(x)-u(y)|^{p}}{|x-y|^{d+sp}} 
 |x|^{\alpha_{1} p} |y|^{\alpha_{2} p}  \, dx \, dy \right)^{\!\frac{1}{p}},
\end{equation}
under the conditions
\begin{equation*}
 0 \leq \alpha < s, \qquad sp- \alpha p < d, \qquad \text{and} \qquad 
   \tau = \frac{dp}{d-sp+ \alpha p}.  
\end{equation*}
The inequality \eqref{frac-sobolev} is nothing but the fractional Sobolev inequality in weighted spaces, with the critical exponent
\begin{equation*}
    p^{*}_{s, \alpha} := \tau = \frac{dp}{d-sp+ \alpha p}, \qquad \text{if} \quad sp- \alpha p < d.
\end{equation*}
We can therefore state the following fundamental result.

\begin{lemma}[Weighted fractional Sobolev inequality]\label{Weighted fractional Sobolev inequality}
    Let $p>1$, $s \in (0,1)$, and $\alpha_{1}, \alpha_{2} \in \mathbb{R}$ be such that $\alpha= \alpha_{1}+ \alpha_{2}$,  $\alpha_{1}p$, $\alpha_{2}p \in (-d, sp)$, $ 0 \leq 
   \alpha p < sp$, $sp-\alpha p < d$
    with $p^{*}_{s, \alpha}= dp/(d-sp+\alpha p)$. Then there exists a constant $\mathcal{C}_{\alpha}=\mathcal{C}_{\alpha}(d,p,s,\alpha_{1}, \alpha_{2})>0$ such that
    \begin{align}
        \left( \int_{\mathbb{R}^{d}} |u(x)|^{p^{*}_{s, \alpha}} \, dx \right)^{\!\frac{1}{p^{*}_{s, \alpha}}}
        \leq \mathcal{C}^{\frac{1}{p}}_{\alpha} \Big( \int_{\mathbb{R}^{d}} \int_{\mathbb{R}^{d}}
        \frac{ |u(x)-u(y)|^{p}}{|x-y|^{d+sp}} 
        |x|^{\alpha_{1} p} |y|^{\alpha_{2} p} &  \, dx \, dy \Big)^{\!\frac{1}{p}} , \nonumber  \\ &  \forall \ u \in W^{s,p, \alpha}(\mathbb{R}^{d}).
    \end{align}
\end{lemma}

\smallskip

We now derive a weighted version of the interpolation inequality, as formulated in Theorem \ref{Theorem: CKN ineq without opt}. For $s \in (0,1)$, $p > 1$, $\alpha= \alpha_{1}+ \alpha_{2}$,  $\alpha_{1}p$, $\alpha_{2}p \in (-d, sp)$, $ 
   \alpha_{1}p+\alpha_{2}p > -d$, $sp-\alpha p < d$, and $q > p$, we define the space $\mathcal{D}^{p,q}_{s, \alpha}$ as
\begin{equation}
   \mathcal{D}^{p,q}_{s, \alpha} = (W^{s,p, \alpha}(\mathbb{R}^{d}) \cap L^{q}(\mathbb{R}^{d})) \setminus \{ 0\}.
\end{equation}
Following the strategy used in the proof of Theorem \ref{Theorem: CKN ineq without opt}, we obtain the following result:

\begin{theorem}\label{Theorem: Weighted CKN ineq without opt}
Let $p>1$, $s \in (0,1)$, and $\alpha_{1}, \alpha_{2} \in \mathbb{R}$ be such that $\alpha= \alpha_{1}+ \alpha_{2}$,  $\alpha_{1}p$, $\alpha_{2}p \in (-d, sp)$, $ 
 0 \leq \alpha p < sp$, and $sp-\alpha p < d$.
Let $q$ be such that  $p < q < \frac{p(d-s+\alpha)}{d-sp+\alpha p}$, and set $r = p \frac{q-1}{p-1}$. Then there exists a constant  $\mathcal{C}_{\alpha}=\mathcal{C}_{\alpha}(d,p,s,\alpha_{1},\alpha_{2})>0$ such that, for all $u \in \mathcal{D}^{p,q}_{s,\alpha}$,
\begin{equation}
   \|u\|_{L^{r}(\mathbb{R}^{d})} \leq 
   \mathcal{C}_{\alpha}^{\frac{a}{p}} 
   [u]_{W^{s,p,\alpha}(\mathbb{R}^{d})}^{a}
   \|u\|_{L^{q}(\mathbb{R}^{d})}^{1-a},
\end{equation}
where $\mathcal{C}_{\alpha}$ is as defined in Lemma~\ref{Weighted fractional Sobolev inequality}, and
\begin{equation}
   a = \frac{d(q-p)}{(q-1)\big(dp-(d-sp+ \alpha p)q\big)},
\end{equation}
with $\delta_{\alpha}:=dp-(d-sp+ \alpha p)q>0$.
\end{theorem}

\begin{proof}[\textbf{Proof of Theorem \ref{Theorem: Weighted Frac log inequality 1}}] Following the strategy of the proof of Theorem \ref{Theorem: Frac log Sob ineq}, and passing to the limit $q \to p$, we obtain that under the normalization $\int_{\mathbb{R}^{d}} |u(x)|^{p} \, dx =1$, the following inequality holds: 
\begin{equation*}
         \int_{\mathbb{R}^{d}} |u(x)|^{p} \log  |u(x)|   \, dx \leq \frac{d}{(s-\alpha)p^{2}} \log \left( \mathcal{C}_{\alpha} \int_{\mathbb{R}^{d}} \int_{\mathbb{R}^{d}}  \frac{|u(x)-u(y)|^{p}}{|x-y|^{d+sp}} |x|^{\alpha_{1} p} |y|^{\alpha_{2} p} \, dx \, dy \right).
    \end{equation*}
  This completes the proof.  
\end{proof}

\begin{remark}
By adapting the proof of Theorem \ref{Theorem:2 Frac log Sob ineq}, we also obtain the following inequality. For any $c>0$ and any $u \in W^{s,p,\alpha}(\mathbb{R}^{d})$,  
\begin{align}
    \int_{\mathbb{R}^{d}} |u(x)|^{p} 
    \log \left( \frac{|u(x)|^{p}}{\|u\|^{p}_{L^{p}(\mathbb{R}^{d})}} \right) dx 
    & + \frac{d}{s-\alpha} (1+ \log c)\, \|u\|^{p}_{L^{p}(\mathbb{R}^{d})} \nonumber \\ 
    & \leq \frac{ d \, e^{p-1} c^{p} \, \mathcal{C}_{\alpha}}{(s-\alpha)p}  
    \int_{\mathbb{R}^{d}} \int_{\mathbb{R}^{d}}  
    \frac{|u(x)-u(y)|^{p}}{|x-y|^{d+sp}} 
    |x|^{\alpha_{1} p} |y|^{\alpha_{2} p} \, dx \, dy.
\end{align}
\end{remark}

\section{Proof of Theorem \ref{Theorem: CKN with opt const}}\label{Proof of CKN with opt}

In this section, we prove Theorem \ref{Theorem: CKN with opt const} with the optimal constant $\mathscr{L}$. We begin by setting up a variational problem, then connect it to the interpolation inequality, and finally show that the infimum value of this variational problem gives the best constant in the inequality.

\smallskip

The next lemma explains the connection between the variational problem and the interpolation inequality with the optimal constant.

\begin{lemma}\label{Lemma: on a,b,M,N}
    For any fixed $a, ~ b, ~ M, ~ N >0$, define $f(s) = s^{a}M + s^{-b} N$, for all $s>0$. Then the minimum value of $f(s)$ is attained at $s= \left( \frac{bN}{aM} \right)^{\frac{1}{a+b}} $, and this minimum value is
    \begin{equation}
        \inf_{s>0} f(s) = f \left( \left( \frac{bN}{aM} \right)^{\frac{1}{a+b}} \right) = \frac{a+b}{a} \left( \frac{b}{a} \right)^{\frac{-b}{a+b}} M^{\frac{b}{a+b}} N^{\frac{a}{a+b}}. 
    \end{equation}
\end{lemma}
\begin{proof}
We first observe that
\begin{equation*}
   \lim_{s \to 0^+} f(s) = \infty \quad \text{and} \quad \lim_{s \to \infty} f(s) = \infty. 
\end{equation*}
Hence the function $f(s)$ must attain its minimum at some point $s>0$. Differentiating, we get
\begin{equation*}
    f'(s) = aMs^{a-1} - bNs^{-b-1}.
\end{equation*}
Setting $f'(s)=0$ gives
\begin{equation*}
    aMs^{a-1} = bNs^{-b-1} \quad \Longrightarrow \quad s^{a+b} = \frac{bN}{aM}.
\end{equation*}
Thus the only critical point is
\begin{equation*}
    s_* = \left(\frac{bN}{aM}\right)^{\tfrac{1}{a+b}}.
\end{equation*}
To check that this corresponds to a minimum, note that for $s \to 0^+$ or $s \to \infty$ we have $f(s)\to\infty$, so the finite critical point $s_*$ must indeed give the global minimum. Finally, substituting $s_*$ into $f(s)$ we obtain
\begin{equation*}
   f(s_*) = M\left(\frac{bN}{aM}\right)^{\tfrac{a}{a+b}}
+ N\left(\frac{bN}{aM}\right)^{-\tfrac{b}{a+b}}
= \frac{a+b}{a}\left(\frac{b}{a}\right)^{-\tfrac{b}{a+b}} 
M^{\tfrac{b}{a+b}} N^{\tfrac{a}{a+b}}. 
\end{equation*}
This proves the claim.
\end{proof}

\smallskip

\begin{proof}[\textbf{Proof of Theorem \ref{Theorem: CKN with opt const}}] Under the assumption of Theorem \ref{Theorem: CKN with opt const}, define the functional
\begin{equation}
    I(u):= \frac{1}{p} \int_{\mathbb{R}^{d}} \int_{\mathbb{R}^{d}} \frac{|u(x)-u(y)|^{p}}{|x-y|^{d+sp}} \, dx \, dy + \frac{1}{q} \int_{\mathbb{R}^{d}} |u(x)|^{q} \, dx.
\end{equation}
Now,  let $K>0$ be a constant, which we fix later. Conside the set
\begin{equation*}
    \Lambda:= \left\{ u \in \mathcal{D}^{p,q}_{s}  : \int_{\mathbb{R}^{d}} |u(x)|^{r} \, dx =K \right\}.
\end{equation*}
From Theorem \ref{Theorem: CKN ineq without opt}, we immediately get
\begin{equation}
    \eta : = \inf_{u \in \Lambda} I(u)>0.
\end{equation}
Next, define the scaling transform for parameters $\alpha:= \frac{d(p-1)}{p(q-1)} >0$ and $\lambda>0$,
\begin{equation*}
     T_{\lambda}(u) = u_{\lambda}(x) = \lambda^{\alpha} u(\lambda x).
\end{equation*}
By change of variables, we find
\begin{equation*}
    \int_{\mathbb{R}^{d}} \int_{\mathbb{R}^{d}} \frac{|u_{\lambda}(x)-u_{\lambda}(y)|^{p}}{|x-y|^{d+sp}} \, dx \, dy = \lambda^{\alpha  p -(d-sp)} \int_{\mathbb{R}^{d}} \int_{\mathbb{R}^{d}} \frac{|u(x)-u(y)|^{p}}{|x-y|^{d+sp}} \, dx \, dy,
\end{equation*}
and 
\begin{equation*}
    \int_{\mathbb{R}^{d}} |u_{\lambda}(x)|^{q} \, dx = \lambda^{-(d- \alpha q)} \int_{\mathbb{R}^{d}} |u(x)|^{q} \, dx.
\end{equation*}
By using the definition of $\alpha$, we have
\begin{equation*}
    \alpha p -(d-sp) = \frac{p(d-s)-q(d-sp)}{q-1}>0 \hspace{3mm} \text{iff} \hspace{3mm} q<\frac{p(d-s)}{d-sp},
\end{equation*}
and 
\begin{equation*}
    d-\alpha q = \frac{d(q-p)}{p(q-1)}>0 \hspace{3mm} \text{as} \hspace{3mm} q>p.
\end{equation*}
It is easy to deduce that for any $u \in \Lambda$, we have $u_{\lambda} \in \Lambda$ for all $\lambda>0$. Consequently, $\inf_{\lambda>0} I(u_{\lambda}) \le I(u_1) = I(u)$, which implies
\begin{equation}\label{eq:one-sided-inf}
    \inf_{u \in \Lambda} \inf_{\lambda>0} I(u_{\lambda})
    \leq
    \inf_{u \in \Lambda} I(u).
\end{equation}
Now set $a= \alpha p-(d-sp)>0$, $b= d-\alpha q>0$, $M= \frac{1}{p} [u]^{p}_{W^{s,p}(\mathbb{R}^{d})}>0$, and $N= \frac{1}{q} \|u\|^{q}_{L^{q}(\mathbb{R}^{d})}>0$. By Lemma \ref{Lemma: on a,b,M,N}, there exists a unique $\lambda_{u}>0$ such that $\inf_{\lambda>0} I(u_{\lambda}) = I(u_{\lambda_{u}})$ for every $u \in \Lambda$. Therefore, using the fact that for any $u \in \Lambda$ we have
$u_{\lambda_{u}} \in \Lambda$, we obtain
\begin{equation*}
   \inf_{u \in \Lambda} I(u)
\leq
I(u_{\lambda_{u}})
=
\inf_{\lambda>0} I(u_{\lambda}), \quad \forall \, u \in \Lambda  \quad \Rightarrow  \quad \inf_{u \in \Lambda} I(u)
    \leq
    \inf_{u \in \Lambda} \inf_{\lambda>0} I(u_{\lambda}).
\end{equation*}
Applying Lemma \ref{Lemma: on a,b,M,N} once again with the above parameters, and using the two estimates obtained above, we conclude that
\begin{align*}
    \inf_{u \in \Lambda} I(u) &= \inf_{u \in \Lambda} \, \inf_{\lambda>0} I(u_{\lambda}) \\ & = \inf_{u \in \Lambda}  \, \inf_{\lambda>0} \left\{ \lambda^{\alpha p -(d-sp)} \frac{1}{p} [u]^{p}_{W^{s,p}(\mathbb{R}^{d})}  +  \lambda^{-(d- \alpha q)} \frac{1}{q} \| u \|^{q}_{L^{q}(\mathbb{R}^{d})} \right\} \\ & = \frac{\alpha (p-q)+sp}{\alpha p - (d-sp)} \left( \frac{d- \alpha q}{\alpha p- (d-sp)} \right)^{\frac{\alpha q-d}{\alpha(p-q)+sp}} \left( \frac{1}{p^{\frac{d- \alpha q}{\alpha (p-q)+sp}}} \right) \left( \frac{1}{q^{\frac{\alpha p - (d-sp)}{\alpha(p-q)+sp}}} \right) \times \\ &  \hspace{5mm} \inf_{u \in \Lambda} [u]^{{\frac{p(d- \alpha q)}{\alpha (p-q)+sp}}}_{W^{s,p}(\mathbb{R}^{d})}  \, \| u \|^{\frac{q(\alpha p - (d-sp))}{\alpha(p-q)+sp}}_{L^{q}(\mathbb{R}^{d})}.
\end{align*}
Let 
\begin{equation*}
    \widetilde{C} := \frac{\alpha (p-q)+sp}{\alpha p - (d-sp)} \left( \frac{d- \alpha q}{\alpha p- (d-sp)} \right)^{\frac{\alpha q-d}{\alpha(p-q)+sp}}\left( \frac{1}{p^{\frac{d- \alpha q}{\alpha (p-q)+sp}}} \right) \left( \frac{1}{q^{\frac{\alpha p - (d-sp)}{\alpha(p-q)+sp}}} \right).
\end{equation*}
Then, we have
\begin{equation*}
    \inf_{u \in \Lambda} I(u) = \widetilde{C} \, \inf_{u \in \Lambda} \, [u]^{{\frac{p(d- \alpha q)}{\alpha (p-q)+sp}}}_{W^{s,p}(\mathbb{R}^{d})} \,  \| u \|^{\frac{q(\alpha p - (d-sp))}{\alpha(p-q)+sp}}_{L^{q}(\mathbb{R}^{d})}.
\end{equation*}
For any $u \in \mathcal{D}^{p,q}_{s}$, define $u_{0}(x) = K^{\frac{1}{r}} \frac{u(x)}{\| u \|_{L^{r}(\mathbb{R}^{d})}}$. Then,
\begin{equation*}
    \int_{\mathbb{R}^{d}} |u_{0}(x)|^{r} \, dx   = K.
\end{equation*}
Therefore, $u_{0} \in \Lambda$. Then, we have
\begin{align*}
   \inf_{u \in \Lambda} \, I(u) & = \widetilde{C} \, \inf_{u \in \mathcal{D}^{p,q}_{s}} \, \left[ \frac{K^{\frac{1}{r}}u}{\| u \|_{L^{r}(\mathbb{R}^{d})}} \right]^{{\frac{p(d- \alpha q)}{\alpha (p-q)+sp}}}_{W^{s,p}(\mathbb{R}^{d})} \,  \norm{ \frac{K^{\frac{1}{r}} u}{\norm{u  \|_{L^{r}(\mathbb{R}^{d})}}} }^{\frac{q(\alpha p - (d-sp))}{\alpha(p-q)+sp}}_{L^{q}(\mathbb{R}^{d})} \\ & = \widetilde{C} \, K ^{\frac{d(p-q)+spq}{r (\alpha(p-q)+sp)}}  \, \inf_{ u \in \mathcal{D}^{p,q}_{s}} \,  \frac{[u]^{{\frac{p(d- \alpha q)}{\alpha (p-q)+sp}}}_{W^{s,p}(\mathbb{R}^{d})} \, \| u \|^{\frac{q(\alpha p - (d-sp))}{\alpha(p-q)+sp}}_{L^{q}(\mathbb{R}^{d})}}{\| u \|^{\frac{d(p-q)+spq}{\alpha(p-q)+sp}}_{L^{r}(\mathbb{R}^{d})}}.
\end{align*}
Next, we fix
\begin{align*}
    K = \left( \frac{1}{\widetilde{C}} \right)^{\frac{r( \alpha(p-q)+sp)}{d(p-q)+spq}} & = \left(\frac{\alpha p-(d-sp)}{\alpha(p-q) +sp} \right)^{\frac{r( \alpha(p-q)+sp)}{d(p-q)+spq}} \left( \frac{\alpha p-(d-sp)}{d-\alpha q} \right)^{\frac{r(\alpha q-d)}{d(p-q)+spq}} \times \\ & \hspace{6mm} p^{\frac{r(d-\alpha q)}{d(p-q)+spq}} \, q^{\frac{r(\alpha p -(d-sp))}{d(p-q)+spq}},
\end{align*}
which implies that
\begin{equation*}
  \eta = \inf_{u \in \Lambda} \, I(u)  = \inf_{ u \in \mathcal{D}^{p,q}_{s}}  \,  \frac{[u]^{{\frac{p(d- \alpha q)}{\alpha (p-q)+sp}}}_{W^{s,p}(\mathbb{R}^{d})} \, \| u \|^{\frac{q(\alpha p - (d-sp))}{\alpha(p-q)+sp}}_{L^{q}(\mathbb{R}^{d})}}{\| u \|^{\frac{d(p-q)+spq}{\alpha(p-q)+sp}}_{L^{r}(\mathbb{R}^{d})}}. 
\end{equation*}
Therefore, we obtain
\begin{equation*}
    \| u \|_{L^{r}(\mathbb{R}^{d})} \leq \left( \frac{1}{\eta} \right)^{\frac{\alpha(p-q)+sp}{d(p-q)+spq}} \, [u]^{\frac{p(d- \alpha q)}{d(p-q)+spq}}_{W^{s,p}(\mathbb{R}^{d})} \, \| u \|^{\frac{q(\alpha p-(d-sp))}{d(p-q)+spq}}_{L^{q}(\mathbb{R}^{d})}, \quad \forall \ u \in \mathcal{D}^{p,q}_{s}.
\end{equation*}
Define $\mathscr{L} := \left( \frac{1}{\eta} \right)^{\frac{\alpha(p-q)+sp}{d(p-q)+spq}}$. From the above analysis, it follows that, $\mathscr{L}$ is the optimal constant in the interpolation inequality. Substituting the value of $\alpha$ in the above inequality, we arrive at 
\begin{equation*}
     \| u \|_{L^{r}(\mathbb{R}^{d})} \leq \mathscr{L} [u]^{a}_{W^{s,p}(\mathbb{R}^{d})} \| u \|^{1-a}_{L^{q}(\mathbb{R}^{d})}, \quad \forall \ u \in \mathcal{D}^{p,q}_{s}.
\end{equation*} 
This proves the theorem.
\end{proof}

\bigskip

\textbf{Acknowledgement:} The author thanks the Department of Mathematics and Statistics at the Indian Institute of Technology Kanpur, India, for providing a supportive research environment. The author also acknowledges the support of the Indian Institute of Technology Kanpur, India, through the FARE (Fellowship for Academic and Research Excellence) fellowship. The author is thankful to Michał Kijaczko for suggesting the extension of logarithmic Sobolev inequalities to the weighted fractional Sobolev space.


\begin{thebibliography}{10}

\bibitem{Adams1979}
R.~A. Adams, \emph{General logarithmic {S}obolev inequalities and {O}rlicz imbeddings}, J. Functional Analysis \textbf{34} (1979), no.~2, 292--303.

\bibitem{Weiwei2022}
W.~Ao, A.~DelaTorre, and M.~del~Mar Gonz\'alez, \emph{Symmetry and symmetry breaking for the fractional {C}affarelli-{K}ohn-{N}irenberg inequality}, J. Funct. Anal. \textbf{282} (2022), no.~11, Paper No. 109438, 58.

\bibitem{Beckner1998}
W.~Beckner, \emph{Geometric proof of {N}ash's inequality}, Internat. Math. Res. Notices (1998), no.~2, 67--71.

\bibitem{Beckner1999}
W.~Beckner, \emph{Geometric asymptotics and the logarithmic {S}obolev inequality}, Forum Math. \textbf{11} (1999), no.~1, 105--137.

\bibitem{Brezis2001}
J.~Bourgain, H.~Brezis, and P.~Mironescu, \emph{Another look at {S}obolev spaces}, Optimal control and partial differential equations, IOS, Amsterdam, 2001, pp.~439--455.

\bibitem{Caffarelli1984}
L.~Caffarelli, R.~Kohn, and L.~Nirenberg, \emph{First order interpolation inequalities with weights}, Compositio Math. \textbf{53} (1984), no.~3, 259--275.

\bibitem{Carlen1991}
E.A. Carlen, \emph{Superadditivity of {F}isher's information and logarithmic {S}obolev inequalities}, J. Funct. Anal. \textbf{101} (1991), no.~1, 194--211.

\bibitem{Ruzhansky2024}
M.~Chatzakou and M.~Ruzhansky, \emph{Revised logarithmic {S}obolev inequalities of fractional order}, Bull. Sci. Math. \textbf{197} (2024), Paper No. 103530, 9.

\bibitem{Lu2021}
L.~Chen, G.~Lu, and C.~Zhang, \emph{Maximizers for fractional {C}affarelli-{K}ohn-{N}irenberg and {T}rudinger-{M}oser inequalities on the fractional {S}obolev spaces}, J. Geom. Anal. \textbf{31} (2021), no.~4, 3556--3582.

\bibitem{Cotsiolis20024}
A.~Cotsiolis and N.K. Tavoularis, \emph{Best constants for {S}obolev inequalities for higher order fractional derivatives}, J. Math. Anal. Appl. \textbf{295} (2004), no.~1, 225--236.

\bibitem{Manuel2003}
M.~Del~Pino and J.~Dolbeault, \emph{The optimal {E}uclidean {$L^p$}-{S}obolev logarithmic inequality}, J. Funct. Anal. \textbf{197} (2003), no.~1, 151--161.

\bibitem{Hitchiker2012}
E.~Di~Nezza, G.~Palatucci, and E.~Valdinoci, \emph{Hitchhiker's guide to the fractional {S}obolev spaces}, Bull. Sci. Math. \textbf{136} (2012), no.~5, 521--573.

\bibitem{Valdinoci2015}
S.~Dipierro and E.~Valdinoci, \emph{A density property for fractional weighted {S}obolev spaces}, Atti Accad. Naz. Lincei Rend. Lincei Mat. Appl. \textbf{26} (2015), no.~4, 397--422.

\bibitem{Edmunds2023}
D.~E. Edmunds and W.~D. Evans, \emph{Fractional {S}obolev spaces and inequalities}, Cambridge Tracts in Mathematics, vol. 230, Cambridge University Press, Cambridge, 2023.

\bibitem{ghosh2023}
S.~Ghosh, V.~Kumar, and M.~Ruzhansky, \emph{Best constants in subelliptic fractional {S}obolev and {G}agliardo-{N}irenberg inequalities and ground states on stratified {L}ie groups}, Calc. Var. Partial Differential Equations \textbf{65} (2026), no.~1, Paper No. 28.

\bibitem{Gross1975}
L.~Gross, \emph{Logarithmic {S}obolev inequalities}, Amer. J. Math. \textbf{97} (1975), no.~4, 1061--1083.

\bibitem{Lieb2001}
E.H. Lieb and M.~Loss, \emph{Analysis}, second ed., Graduate Studies in Mathematics, vol.~14, American Mathematical Society, Providence, RI, 2001.

\bibitem{Mazya2002}
V.~Maz'ya and T.~Shaposhnikova, \emph{On the {B}ourgain, {B}rezis, and {M}ironescu theorem concerning limiting embeddings of fractional {S}obolev spaces}, J. Funct. Anal. \textbf{195} (2002), no.~2, 230--238.

\bibitem{Squassina2018}
H.-M. Nguyen and M.~Squassina, \emph{Fractional {C}affarelli-{K}ohn-{N}irenberg inequalities}, J. Funct. Anal. \textbf{274} (2018), no.~9, 2661--2672.

\bibitem{Talenti1976}
G.~Talenti, \emph{Best constant in {S}obolev inequality}, Ann. Mat. Pura Appl. (4) \textbf{110} (1976), 353--372.

\bibitem{Toscani1997}
G.~Toscani, \emph{Sur l'in\'egalit\'e{} logarithmique de {S}obolev}, C. R. Acad. Sci. Paris S\'er. I Math. \textbf{324} (1997), no.~6, 689--694.

\bibitem{Weissler}
F.B. Weissler, \emph{Logarithmic {S}obolev inequalities for the heat-diffusion semigroup}, Trans. Amer. Math. Soc. \textbf{237} (1978), 255--269.

\end{thebibliography}

\end{document}